\documentclass[11pt]{article}
\PassOptionsToPackage{obeyspaces}{url}
\usepackage{amsfonts, amsmath, amssymb}

\usepackage[hidelinks]{hyperref}
\usepackage{bookmark}
\bookmarksetup{open,numbered,depth=5} 

\usepackage{amsthm}

\usepackage{breakurl}
\usepackage{tikz}

\usepackage{etoolbox} 
\patchcmd{\thebibliography}{\leftmargin\labelwidth}{\leftmargin\labelwidth\addtolength\itemsep{-0.1\baselineskip}}{}{}

\usepackage{mathtools}

\usepackage{epstopdf}

\usepackage[justification=raggedright]{caption}

\author{Boris Bukh\thanks{Department of Mathematical Sciences, Carnegie Mellon University, Pittsburgh, PA 15213, USA\@. Supported in part by U.S.\ taxpayers through NSF CAREER grant DMS-1555149.}\and
R. Amzi Jeffs\thanks{Department of Mathematical Sciences, Carnegie Mellon University, Pittsburgh, PA 15213, USA\@. Supported by the National Science Foundation through Award No. 2103206.}}

\title{Enumeration of interval graphs and $d$-representable complexes}
\date{March 2022}

\usepackage[nameinlink]{cleveref}

\newtheorem{theorem}{Theorem}
\newtheorem*{collapsethmrestated}{\Cref{thm:d-collapsible} (restated)}

\newtheorem{corollary}[theorem]{Corollary}
\newtheorem{proposition}[theorem]{Proposition}

\theoremstyle{remark}
\newtheorem{example}[theorem]{Example}

\newcommand*{\eqdef}{\stackrel{\mbox{\normalfont\tiny def}}{=}}  
\newcommand*{\R}{\mathbb{R}}                                     
                                
\DeclareMathOperator{\conv}{conv}                                 
\DeclareMathOperator{\nerve}{nerve}                                 

\begin{document}
\maketitle
\begin{abstract}
For each fixed $d\ge 1$, we obtain asymptotic estimates for the number of $d$-representable simplicial complexes on $n$ vertices as a function of~$n$.
The case $d=1$ corresponds to counting interval graphs, and we obtain new results in this well-studied case as well.
Our results imply that the $d$\nobreakdash-representable complexes comprise a vanishingly small fraction of $d$-collapsible complexes.
\end{abstract}

\section{Introduction}

Given a tuple $\mathcal C =  (C_1, C_2, \dotsc, C_n)$ of convex sets, one may record their intersection pattern using the \emph{nerve complex}
\[
\nerve(\mathcal C) \eqdef \Big\{\sigma \, \Big\vert\, \bigcap_{i\in\sigma} C_i \neq \emptyset \Big\}. 
\]
Observe that $\nerve(\mathcal C)$ is a simplicial complex on the vertex set $[n]\eqdef \{1,2,\dotsc,n\}$.

A simplicial complex is called \emph{$d$-representable} if it is the nerve of a tuple of convex sets in $\R^d$. 
Such a tuple is called a \emph{$d$-representation} of $\Delta$. 
Given a $d$-representation, one may fix a point $p_\sigma$ in each nonempty region $\bigcap_{i\in\sigma} C_i$ and then replace each $C_i$ by the convex hull of the $p_\sigma$ it contains, obtaining a $d$-representation of the same complex in which each set is compact.  
Henceforth we will only consider $d$-representations consisting of compact convex sets.  

The class of $d$-representable complexes enjoys many useful topological and combinatorial properties. 
In particular, Helly's theorem implies that a $d$-representable complex is completely determined by its $d$-skeleton.
Indeed, if $\Delta$ is $d$-representable and $\sigma\subseteq [n]$ is a set of size $d+2$ or larger, then $\sigma$ is a face of $\Delta$ if and only if every subset of $\sigma$ with size $d+1$ is a face of $\Delta$.

Perhaps the closest combinatorial analog of $d$-representability is \emph{$d$-collapsibility}. 
A \emph{free face} in a simplicial complex is a face that is contained in a unique facet.
A \emph{$d$-collapse} is the operation of deleting a free face with dimension $d-1$ or less (and all faces that contain it). 
Finally, a simplicial complex $\Delta$ is called \emph{$d$-collapsible} if there is a sequence of $d$-collapses from $\Delta$ to the empty complex. 

In 1975 Wegner \cite{wegner_collapsing} proved that $d$-representable complexes are $d$-collapsible by ordering the regions corresponding to faces according to a generic linear function (see \cite[Figure 3]{tancer_survey} for a modern sketch of the proof). There are $d$-collapsible complexes that are not $d$-representable, for any $d\ge 1$,
and there has been a variety of work exploring the properties of $d$-representable complexes, $d$-collapsible complexes and the related (even more general) class of $d$-Leray complexes \cite{kalai_intersection_patterns, eckhoff_vectors, matousek_tancer_gap, tancer_projective_planes, tancer_np_complete}. 
Tancer's survey \cite{tancer_survey} provides a good overview of results in this area prior to~2013. 

In this work we study, for fixed $d\ge 1$, the growth of the number of $d$-representable complexes on the vertex set $[n]$ as a function of $n$.
We also consider the number of $d$-collapsible complexes. 
For each $d\ge 1$ we determine the rough rate of growth for these quantities and provide upper and lower bounds on the constants involved.

\paragraph{One-dimensional case (counting interval graphs).}
We first discuss the case $d=1$, which is of particular interest.
Since the $1$\nobreakdash-skeleton of a $1$\nobreakdash-representable complex determines the entire complex, $1$-representable complexes are exactly the clique complexes of the well-studied class of \emph{interval graphs}, which record the nonempty pairwise intersections of $n$ closed intervals\footnote{There seems to be no consensus as to whether the empty set is an interval. However,
the class of interval graphs is unaffected by the choice of a side in this question, for one may always replace each empty set in a $1$-representation by an interval that intersects no other.} on the real line. 
Thus the number of $1$-representable complexes on the vertex set $[n]$ is equal to the number of (labeled) interval graphs on this vertex set. 

There were several works \cite{cohen_komlos_mueller_probability,hanlon,gavoille_paul_labeling,yang_pippenger,acan_chakraborty_jo_satti_succinct} concerned with estimating the number of interval graphs.
The currently sharpest asymptotic result is due to Gavoille and Paul \cite{gavoille_paul_labeling} who proved that the number of interval graphs on $n$ vertices is $\exp\bigl(2n\log n-O(n\log\log n)\bigr)$.
The upper bound is straightforward, as an $n$-vertex interval graph is determined by the relative order of $2n$ interval endpoints in its $1$\nobreakdash-representation.
The nontrivial part of their result is in the proof of the lower bound. Unaware of this result, weaker lower bounds
were subsequently found by Yang and Pippenger \cite{yang_pippenger} and Acan, Chakraborty, Jo and Satti~\cite{acan_chakraborty_jo_satti_succinct}.
A different approach was taken by Hanlon \cite{hanlon} who found a recursive expression for the generating function for the number of \emph{unlabeled}
interval graphs. In addition, there were several works \cite{ysku,mikos} on quick algorithmic enumeration of non-isomorphic interval graphs.

By relating the enumeration of interval graphs to that of \emph{interval orders} and applying results of Brightwell and Keller \cite{brightwell_keller},
we find an improved asymptotic for the number of the interval graphs.
\begin{theorem}\label{thm:interval}
  The number of interval graphs on the vertex set $[n]$ is $e^{2n\log n - (2+\log (\pi^2/6))n +O(\log n)}$.
\end{theorem}

\paragraph{Higher-dimensional case.} Our results in dimension $d\geq 2$ are less precise than for~$d=1$.
\begin{theorem}\label{thm:complex}
  The number  $f_d(n)$ of $d$-representable complexes on the vertex set $[n]$ satisfies
  \[
  \tfrac{2}{d^d}n^d\log n - O(n^d) \le \log f_d(n) \le \tfrac{2}{(d-1)!} n^d\log n + O(n^d). 
  \]
  In particular, $f_d(n) = e^{\Theta(n^d\log n)}$. 
\end{theorem}

It turns out that it is not difficult to estimate the number of $d$-collapsible complexes on the vertex set~$[n]$. We obtain the following result.
\begin{theorem}\label{thm:d-collapsible}
The number $g_d(n)$ of $d$-collapsible complexes on the vertex set $[n]$ satisfies
  \[
  \tfrac{1}{(d+1)^{d+1}}n^{d+1} + O(n^d) \le \log_2 g_d(n) \le \tfrac{1}{(d+1)!} n^{d+1} + O(n^d). 
  \]
In particular, $g_d(n) = e^{\Theta(n^{d+1})}$. 
\end{theorem}
This implies that there are far more $d$-collapsible complexes than $d$-representable complexes. That is surprising
because nearly all published results for $d$-representable complexes in fact apply to $d$-collapsible complexes. 

Another curious consequence of these estimates is that there are more $d$-representable complexes than $(d-1)$-collapsible complexes.
That this is not obvious can be seen from the work of Tancer \cite{tancer_projective_planes}
who, for any $d$, described complexes that are $2$-collapsible but not~$d$-representable.

\paragraph{Thanks.} We are grateful to the referee who made a number of useful suggestions. We are especially grateful to them
for catching a subtle off-by-one mistake in the proof of \Cref{thm:order-sandwich}. We also thank Jacob Fox and J\'anos Pach
for useful comments.

\section{Representable complexes in general: proof of \texorpdfstring{\Cref{thm:complex}}{Theorem 2}}\label{sec:complex}
We prove \Cref{thm:complex} modulo the case $d=1$, which is established in \Cref{sec:interval} by proving \Cref{thm:interval}.
\paragraph{Upper bound.} 
Let $(C_1, C_2, \ldots, C_n)$ be a $d$-representation of a complex $\Delta$ on the vertex set $[n]$. 
Define $N \eqdef \binom{n}{d}$.
Let $\pi$ be a linear projection to the span of $e_1$ in $\mathbb R^d$, and for $\sigma\in \binom{[n]}{d}$ define $I_\sigma \eqdef\nobreak \pi(\bigcap_{i\in \sigma} C_i)$. 
Observe that each $I_\sigma$ is a (possibly empty) closed interval.
We may regard the collection of intervals $\{I_\sigma \mid \sigma\in \binom{[n]}{d}\}$ as a $1$-representation of a simplicial complex $\Gamma$ on the vertex set~$[N]$.

We claim that $\Gamma$ completely determines the $d$-dimensional faces of $\Delta$.
Specifically, we claim that for any $\tau \in \binom{[n]}{d+1}$ we have $\tau\in \Delta$ if and only if the intervals $\{I_{\tau\setminus \{i\}} \mid i \in \tau\}$ share a common point. 
One direction is clear: if $p\in \bigcap_{i\in \tau} C_i$ then $\pi(p)$ lies in all the appropriate intervals.
For the reverse inclusion, suppose that the intervals share a common point, and consider the fiber of $\pi$ over this point.
This fiber is a $(d-1)$-dimensional affine subspace of $\mathbb R^d$. 
Moreover, all $d$-fold intersections of the collection $\{C_i \mid i \in\tau\}$ contain a point in this fiber.
Applying Helly's theorem inside the fiber (regarded as a copy of $\mathbb R^{d-1}$), we conclude that the fiber contains a point in $\bigcap_{i\in\tau} C_i$.
Thus $\tau \in \Delta$ as desired.

Helly's theorem (now applied in $\mathbb R^d$) implies that the $d$-skeleton of $\Delta$ determines all of $\Delta$.
The faces dimension $d$ and $d-1$ in $\Delta$ are determined by $\Gamma$, and the remainder of the $d$-skeleton is determined by choosing faces of dimension less than $d-1$. 
Thus we have 
\[
f_d(n) \le f_1(N) \cdot 2^{\binom{n}{<d}},
\]
where $\binom{n}{<d}\eqdef\binom{n}{d-1}+\dotsb+\binom{n}{0}$.
Taking logarithms and using \Cref{thm:interval}, we obtain
\[
\log(f_d(n)) \le  2 N \log N + O(N) + O(n^{d-1}) = \frac{2}{(d-1)!} n^d\log n + O(n^d). 
\]

\paragraph{Lower bound.} 
Below, we shall use the notation $\Delta \ast w$ to denote the \emph{cone over $\Delta$ with apex $w$}, which is the simplicial complex whose facets are obtained by adding a new vertex $w$ to every facet of~$\Delta$.

\begin{proposition}\label{prop:split-construction}
Let $V$ and $W$ be disjoint sets. 
For each $w\in W$, let $\Delta_w$ be a $(d-1)$\nobreakdash-representable complex on vertex set~$V$. 
Then the simplicial complex 
\[
\Delta \eqdef 2^V \cup \bigcup_{w\in W} (\Delta_w\ast w)
\]
on vertex set $V\sqcup W$ is $d$-representable.
\end{proposition}
\begin{proof}
Pick a full-dimensional polytope $P$ in $\mathbb R^{d}$ with at least $|W|$ many facets, and choose a point~$p$ in the interior of~$P$. 
Let $\{C_w\mid w\in W\}$ be a collection of $(d-1)$-dimensional simplices, each contained in the interior of distinct facets of $P$.
For each $w\in W$, let $\{A_{w,v} \mid v\in V\}$ be a $(d-1)$-representation of $\Delta_w$ contained in $C_w$. 
Lastly, for each $v\in V$, define
\[
C_v \eqdef \conv\left(\{p\} \cup \bigcup_{w\in W} A_{w,v}\right).
\]

We claim that the collection $\{C_v\mid v\in V\}\cup \{C_w\mid w\in W\}$ is a $d$-representation of $\Delta$. 
Clearly all $C_v$ share a common point, namely $p$, and all $C_w$ are mutually disjoint from one another.
Thus it suffices to argue for each $w\in W$ and $\sigma\subseteq V$ that $C_w$ contains a point in $\bigcap_{v\in \sigma} C_v$ if and only if $\sigma$ is a face of $\Delta_w$. 
This follows by construction, since the intersection of $\{C_v\mid v\in V\}$ with $C_w$ is exactly the $(d-1)$-representation $\{A_{w,v}\mid v\in V\}$ of $\Delta_w$. 
\end{proof}

\begin{example}\label{ex:split-construction}
\Cref{fig:split-construction} shows the construction in \Cref{prop:split-construction} for the case $V = \{1,2,3\}$, $W =\nobreak \{4,5,6,7\}$ and $d=2$. 
Here each vertex in $W$ ``picks out'' a $1$-representable complex on the vertex set $\{1,2,3\}$. 
Explicitly, $\Delta_4$ is three isolated vertices, $\Delta_5$ is a $2$-simplex, $\Delta_6$ is the $1$-simplex on $\{2,3\}$, and $\Delta_7$ has facets $\{1,3\}$ and $\{2,3\}$.  
In the figure each $C_w$ is slightly thickened and shown in black, while $C_1, C_2,$ and $C_3$ are slightly transparent. 
\begin{figure}
\[
\includegraphics{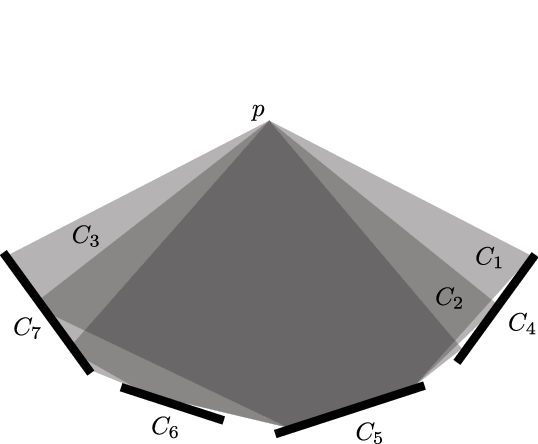}
\]
\caption{The construction used in \Cref{prop:split-construction}.}\label{fig:split-construction}
\end{figure}
\end{example}

\begin{proposition}\label{prop:split-enumeration}
For any $0\leq m\leq n$ we have $f_{d}(n) \ge f_{d-1}(m)^{n-m}$. 
\end{proposition}
\begin{proof}
We will create $f_{d-1}(m)^{n-m}$ many distinct $d$-representable complexes on the vertex set~$[n]$. 
Define $V$ to be the first $m$ elements of $[n]$, and let $W$ be the last $n-m$ elements of~$[n]$. 
For each $w\in W$, let $\Delta_w$ be a $(d-1)$-representable complex on vertex set~$V$. 
Observe that we may recover $\Delta_w$ from the $d$-representable complex $\Delta$ formed in \Cref{prop:split-construction} (in particular, $\Delta_w$ is the link of the vertex $w$ in this complex). 
Thus every collection $\{\Delta_w\mid w\in W\}$ of $(d-1)$-representable complexes determines a unique $d$-representable complex on the vertex set~$[n]$. 
We have $f_{d-1}(m)$ choices for each $\Delta_w$, and since we have $n-m$ many vertices in $W$ the bound follows. 
\end{proof}

\begin{corollary}\label{cor:lower-bound}
For every $d\ge 1$ and sufficiently large $n$ we have
\[
\log f_d(n) \ge \frac{2}{d^d} n^d\log n + O(n^d). 
\]
\end{corollary}
\begin{proof}
We proceed by induction on $d$.
The base case $d=1$ follows from \Cref{thm:interval}, which we will establish in \Cref{sec:interval}. 
Since the function $f_d(n)$ is increasing in $n$, it suffices to consider the case when $n$ is a multiple of~$d$. 
For the inductive step, let $m=\frac{d-1}{d}n$. 
Using \Cref{prop:split-enumeration} and the inductive hypothesis, we compute  
\begin{align*}
\log f_d(n) &\ge \log f_{d-1}(m)^{n-m}\\
& = \tfrac{n}{d}\log f_{d-1}(\tfrac{d-1}{d} n)\\
& \ge \tfrac{n}{d} \left(\tfrac{2}{(d-1)^{d-1}}  (\tfrac{d-1}{d} n)^{d-1}\log(\tfrac{d-1}{d} n) + O\bigl(n^{d-1}\bigr)\right)\\
& = \tfrac{n}{d} \cdot \tfrac{2}{(d-1)^{d-1}}\cdot\tfrac{(d-1)^{d-1}}{d^{d-1}} \cdot n^{d-1} \log n  + O(n^d)\\
& = \tfrac{2}{d^d} n^d\log n + O(n^d).\qedhere
\end{align*}
\end{proof}

\paragraph{Better lower bound in dimension~\texorpdfstring{$2$}{2}.}
The choice of the constant $\frac{d-1}{d}$ in the definition of $m$ above is optimal, and so $2/d^d$ is the best constant that can be deduced from the construction in \Cref{prop:split-enumeration}.
However, in \Cref{prop:beta2} below we will use a slightly different construction to obtain the constant $\frac{3}{4}$ when $d=2$, which is better than 
the constant $\frac{2}{2^2} = \frac{1}{2}$ from \Cref{cor:lower-bound}. One can use the inductive argument in \Cref{cor:lower-bound} to
prove the lower bound of $\frac{3}{d^d} n^d\log n+O(n)$, for $d\geq 2$, by using \Cref{prop:beta2} as the base case.
It would be interesting to determine the constant $\beta_2$ such that $\log f_2(n) = \beta_2n^2\log n + O(n)$. 

\begin{proposition}\label{prop:beta2}
For sufficiently large $n$ we have $f_2(n)\ge f_1(n/4)^{3n/2}$. 
In particular, it follows that $\log f_2(n) \ge \frac{3}{4} n^2\log n + O(n^2)$ for sufficiently large $n$. 
\end{proposition}
\begin{proof}
Since $f_2(n)$ is increasing in $n$, it suffices to prove the result when $n$ is a multiple of $4$, say $n=4m$.
We wish to construct $f_1(n/4)^{3n/2} = f_1(m)^{6m}$ distinct $2$-representable complexes. 
Our starting point is an arrangement of closed convex sets $W$, $X$, $Y$ and $Z$ in the plane, shown in \Cref{fig:beta2}. 
These sets have disjoint interiors, and any pair shares a line segment along the boundary.
These line segments are labeled $A$, $B$, $C$, $D$, $E$ and $F$ in the figure.

\begin{figure}
\[
\includegraphics{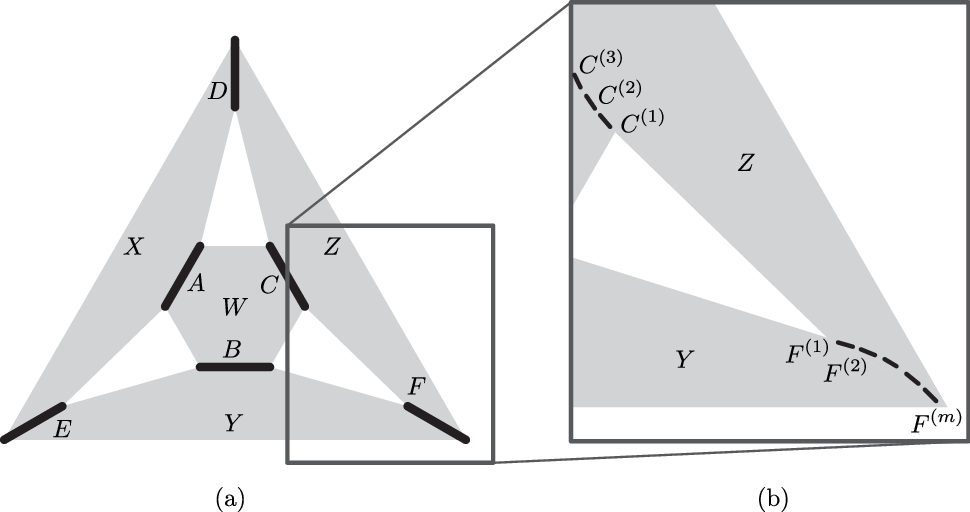}
\]
\caption{(a) The convex regions $W$, $X$, $Y$ and $Z$ in $\R^2$ used in the proof of \Cref{prop:beta2}.\\\hskip47.99pt (b) The small secants chosen near each arc.}\label{fig:beta2}
\end{figure}

Consider replacing each segment by a smooth arc, oriented so that the arc bends closer to the corresponding label in the figure. 
Choose $m$ arbitrarily small disjoint line segments $A^{(1)}, A^{(2)}, \ldots, A^{(m)}$ whose endpoints lie on the arc~$A$.
Note that these segments are in convex position, in the sense that they form a subset of the edges of a convex polygon.
Similarly, choose $m$-many small disjoint line segments in convex position with endpoints along $B$, $C$, $D$, $E$ and $F$. 
We thus obtain $6m$ line segments in total. Since these line segments are secants to the arcs $A$ through $F$, respectively,
we call them simply \emph{secants}.

By making our arcs sufficiently close to the original line segments and making the secants sufficiently small, we can guarantee the following properties:
\begin{itemize}
\item[(i)] For any $i\in [m]$, the convex hull of $A^{(i)}\cup B^{(i)}\cup C^{(i)}$ does not contain a point in any secant except for $A^{(i)}$, $B^{(i)}$ and~$C^{(i)}$.
\item[(ii)] For any $i\in [m]$, the secant $D^{(i)}$ together with all $A^{(j)}$s and all $E^{(k)}$s forms a collection of $2m+1$ disjoint line segments in convex position. 
 \item[(iii)] Condition (ii) holds, but with $(D, A, E)$ replaced by the triple $(E, B, F)$. Similarly, it also holds with $(D,A,E)$ replaced by $(F, C, D)$. 
\end{itemize}
Now, for every $i\in [m]$, fix a tuple $(\Delta_{A,i}, \Delta_{B,i},\ldots, \Delta_{F,i})$ of $1$-representable complexes on the vertex set~$[m]$. 
Place a $1$-representation of $\Delta_{A,i}$ inside the secant~$A^{(i)}$.
Let $A^{(i)}_1,\dotsc,A^{(i)}_{m}\subset A^{(i)}$ be the line segments in the representation
of~$\Delta_{A,i}$. Likewise, place $1$-representations of all other complexes inside their corresponding secants, denoting them
by $B^{(i)}_{j},C^{(i)}_{j}$ and so forth. 
We obtain $6m$ $1$-representations arranged in the plane, with $m$-many near each of the segments $A,B,C, D, E$ and~$F$. 

We now define a $2$-representation of a complex $\Delta$ on the vertex set 
\[
\{w_1,\ldots, w_m, x_1, \ldots, x_m, y_1,\ldots, y_m,z_1,\ldots, z_m\},
\]
which can be regarded as a relabeling of $[n]$. 
For any $i\in [m]$, define $C_{w_i} \eqdef \conv(A^{(i)}\cup B^{(i)}\cup C^{(i)})$. 
Moreover, define convex sets 
\begin{align*}
  C_{x_i}&\eqdef \conv \Bigl(D^{(i)}\cup \bigcup_{j=1}^m A^{(j)}_{i} \cup \bigcup_{j=1}^m E^{(j)}_{i}\Bigr),\\
  C_{y_i}&\eqdef \conv \Bigl(E^{(i)}\cup \bigcup_{j=1}^m B^{(j)}_{i} \cup \bigcup_{j=1}^m F^{(j)}_{i}\Bigr),\\
  C_{z_i}&\eqdef \conv \Bigl(F^{(i)}\cup \bigcup_{j=1}^m C^{(j)}_{i} \cup \bigcup_{j=1}^m D^{(j)}_{i}\Bigr).
\end{align*}

Now, by properties (i) and (ii) above, the sets $C_{x_1}\cap C_{w_i},\dotsc,C_{x_m}\cap C_{w_i}$ form a $1$-representation~$\Delta_{A,i}$.
Thus we may recover $\Delta_{A,i}$ from $\Delta$: it is the link of $w_i$, restricted to the vertex set $\{x_1,\ldots, x_m\}$. 
Similar reasoning allows us to recover $\Delta_{B,i},\ldots, \Delta_{F,i}$ (for example, $\Delta_{F,i}$ is the link of $z_i$ restricted to the vertex set $\{y_1,\ldots, y_m\}$). 

Thus every choice of $m$-many $6$-tuples $(\Delta_{A,i}, \Delta_{B,i},\ldots, \Delta_{F,i})$ of $1$-representable complexes on the vertex set~$[m]$ yields a unique $2$-representable complex on vertex set~$[n]$. 
There are $f_1(m)^{6m}$ ways to choose $m$-many such $6$-tuples, and so the first bound follows.
The bound $\log f_2(n) \ge \frac{3}{4}n^2\log n + O(n)$ follows by taking logarithms and applying \Cref{thm:interval}. 
\end{proof}

\section{Interval graphs and interval orders: Proof of \texorpdfstring{\Cref{thm:interval}}{Theorem 1}}\label{sec:interval}

To prove \Cref{thm:interval} we will relate the number of interval graphs to the number of interval orders. 
Given a tuple $(I_1, \ldots, I_n)$ of closed intervals, its \emph{interval order} is the partial order on $[n]$ where $i < j$ if and only if $I_i$ is completely to the left of $I_j$. 
Observe that one may recover the interval graph associated to a tuple from the interval order.
Indeed, $i$ and $j$ are incomparable if and only if their intervals overlap, and so the interval graph is the incomparability graph of the interval order. 

Surprisingly, it is possible to go in the opposite direction: to turn an interval graph into an interval order.
That leads to to the following result, which shows that $f_1(n)$, the number of interval graphs,
is very close to the number of interval orders.
\begin{theorem}[Sandwich theorem]\label{thm:order-sandwich}
Let $g(n)$ denote the number of interval orders on $[n]$. 
Then
\[
g(n-2)\le f_1(n) \le g(n). 
\]
\end{theorem}

The upper bound in Theorem \ref{thm:order-sandwich} is immediate.
For the lower bound, we construct an injective map \[
\{\text{Interval orders on $[n-2]$}\}\,\, \longrightarrow\,\, \{\text{Interval graphs on $[n]$}\}.
\]
\paragraph{Compressed representations.} 
To start, we fix a convenient representation for each interval order on~$[n-2]$.
Namely, a representation is \emph{compressed} if the following three conditions hold:
\begin{itemize}
\item[(i)] Consecutive left endpoints (respectively, right endpoints) are the same.
\item[(ii)] The endpoints take values in $[2n-4]$. 
\item[(iii)] The subset of integers in $[2n-4]$ that are endpoints of some interval form a consecutive sequence starting at 1. 
\end{itemize}
Note that we can easily obtain a compressed representation from any representation by first sliding consecutive endpoints of the same type together, then applying a monotone rescaling and shift to achieve (ii) and (iii), none of which will change the interval order being represented.
We will regard a compressed representation as a pair of functions $L\colon [n-2]\to [2n-4]$ and $R\colon [n-2]\to [2n-4]$, where $L(i)$ (respectively, $R(i)$) is the coordinate of the left (respectively, right) endpoint of the $i$-th interval. 
In a compressed representation, condition (iii) guarantees that left endpoints only occur at odd values, while right endpoints only occur at even values. 
With this in mind, it will also be convenient to regard a compressed representation as a pair of lists $\mathcal A = (A_1, A_2, \ldots, A_{n-2})$ and $\mathcal B = (B_1, B_2, \ldots, B_{n-2})$ where
\begin{align*}
A_i &\eqdef \{j\in[n-2]\mid L(j) = 2i-1\},\text{ and}\\
B_i &\eqdef\{j\in[n-2] \mid R(j) = 2i\}. 
\end{align*}
In other words, $\mathcal A$ and $\mathcal B$ record the different endpoints which appear at the various coordinates in $[2n-4]$.
Note that we may determine $(L,R)$ from $(\mathcal A, \mathcal B)$ and vice versa. 
Also note that because of condition (iii), if $A_i$ is empty, then so are $B_i$ and $A_{i+1}$, and if $B_i$ is empty, then so are $A_{i+1}$ and $B_{i+1}$. 

\paragraph{The injective map.}
Once we have fixed a compressed representation for every interval order on~$[n-2]$, we are ready to encode interval orders into interval graphs. 
Let $(L, R)$ be a compressed representation of an interval order on $[n-2]$, and let $m \eqdef \max\{R(i)\mid i\in[n-2]\}$. 
We define new intervals $J_1, J_2, \ldots, J_n$ where 
\begin{align*}
J_i &\eqdef [L(i), R(i)+2] \quad\text{ for $i\in[n-2]$,} \\
J_{n-1}&\eqdef [0,1]\text{, and}\\
J_n &\eqdef [m+1, m+2]. 
\end{align*} 
Finally, let $G$ be the interval graph represented by $\mathcal J = (J_1, J_2, \ldots, J_n)$. 
We claim that we can recover $(L, R)$ from $G$. 
In fact, we will recover $(\mathcal A, \mathcal B)$ iteratively.

For any $i\in[n-2]$, let $G_i$ be the graph obtained from $G$ by deleting $n-1$, and all vertices in $\bigcup_{j=1}^{i-1} B_j$. 
In particular, $G_1$ is the result of deleting only the vertex $n-1$.
Geometrically, $G_i$ is the interval graph obtained by deleting all intervals in $\mathcal J$ which lie strictly to the left of $2i+1$, except possibly $J_n$.

\begin{proposition}
Fix $i\in[n-3]$, and let $A$ be the set of vertices in  $\bigcup_{j=1}^{i}A_j$ which appear in $G_i$. 
For each $v\in A$, let $N(v)$ be the closed neighborhood of $v$ in $G_i$. 
If $n$ is not the only vertex in $G_i$, then \begin{itemize}
\item[(i)] $A$ is not empty,
\item[(ii)] the various $N(v)$ for $v\in A$ are totally ordered by containment, so there is a unique inclusion-minimal such neighborhood $N$, 
\item[(iii)] $B_i$ is the set of vertices $v\in A$ with $N(v) = N$, and 
\item[(iv)] $A_{i+1}$ is equal to $N\setminus (A\cup\{n\})$. 
\end{itemize}
\end{proposition}
\begin{proof}
To prove item (i), observe that since $G_i$ contains a vertex not equal to $n$, the set $B_i$ must be nonempty. 
Any $v\in B_i$ is a vertex of $G_i$, and has $L(v) < R(v)  = 2i$, which implies $v\in A$.  

For item (ii), observe that $A$ is exactly the set of $v\in [n-2]$ with $L(v)\le 2i-1$ and $R(v)\ge 2i$. 
Given $v\in A$, the right endpoint of $J_v$ is $R(v)+2$.
Among the various vertices $u\in[n-2]$ that appear in $G_i$, we see that $J_v$ intersects $J_u$ if and only if $R(v) + 2 > L(u)$. 
This happens if and only if $u\in A$ or $u\in A_j$ for some $j$ satisfying $i+1\le j\le (R(v)+2)/2$. 
Consequently, for each $v\in A$ the set $N(v)$ is exactly $A\cup \bigcup_{j={i+1}}^{j=(R(v)+2)/2} A_j$, where we use the convention $A_{(m+2)/2} \eqdef \{n\}$ (due to intersection with $J_n$). 
Thus the various neighborhoods of vertices in $A$ are totally ordered by containment, with strictly larger neighborhoods arising from strictly larger values of $R(v)$. This proves (ii).

From the above observations, we have $N(v) = N$ if and only if $R(v) = 2i$, that is, if and only if $v\in B_i$. 
At the beginning of the proof we observed that if $v\in B_i$ then $v\in A$. 
This proves (iii).
Finally, the formula above implies that $N = A \cup A_{i+1}$, plus possibly the vertex $n$. 
This union is disjoint, so $N\setminus (A\cup\{n\}) = A_{i+1}$, proving (iv).
\end{proof}
The proposition above implies that if we know $G$, $A_1, \ldots, A_i$, and $B_1,\ldots, B_{i-1}$ (and hence also~$G_i$), then we may compute $A_{i+1}$ and $B_i$. 
Note that $A_1$ is equal to the neighborhood of $n-1$ in $G$.
Thus, given only $G$, we may first compute $A_1$, and then apply the proposition repeatedly to obtain $A_1,\ldots, A_{n-2}$ and $B_1,\ldots, B_{n-3}$.
Finally, $B_{n-2}$ will simply consist of the vertices missing from all of the already-computed $B_i$.
Thus we recover $(\mathcal A, \mathcal B)$ from $G$, and hence recover our original interval order.
This proves that our map from interval orders to interval graphs is injective, as desired,
concluding the proof of \Cref{thm:order-sandwich}.

\paragraph{Proof of Theorem \ref{thm:interval}.}
Brightwell and Keller \cite[Theorem~6]{brightwell_keller}, building upon earlier works of Zagier \cite{zagier} and Bousquet-M\'elou, Claesson, Dukes and Kitaev \cite{bcdk}, proved that the number $g(n)$ of interval orders on $[n]$ satisfies 
\[
g(n) \sim (n!)^2\sqrt{n}\left(\frac{6}{\pi^2}\right)^n \left( E_0 + \frac{E_1}{n} + \frac{E_2}{n^2} + \cdots \right)
\]
where the $E_i$ are constants. 
By applying Stirling's approximation and simplifying, we obtain
\[
g(n) = e^{2n\log n - (2+ \log(\pi^2/6))n +O(\log n)}.
\]
Replacing $n$ by $n-2$, we see that the same asymptotic holds for $g(n-2)$. 
Thus by Theorem \ref{thm:order-sandwich} we obtain the same asymptotic equality for $f_1(n)$, proving Theorem \ref{thm:interval}. 

\section{Collapsible complexes: proof of \texorpdfstring{\Cref{thm:d-collapsible}}{Theorem 3}}
In the proof of \Cref{thm:complex} we constructed many $d$-representable complexes by starting with 
the case $d=1$ and using induction on $d$. For the inductive step, we merged several $(d-1)$-representable
complexes on the same vertex set into a single $d$-representable complex. It turns out that the very same
operation turns a collection of $(d-1)$-collapsible complexes into a single $d$-collapsible complex.

The following is the direct analogue of \Cref{prop:split-construction} for collapsible complexes.
\begin{proposition}\label{prop:collapsible-split-construction}
Let $V$ and $W$ be disjoint sets. 
For each $w\in W$, let $\Delta_w$ be a $(d-1)$\nobreakdash-collapsible complex on vertex set~$V$. 
Then the simplicial complex 
\[
\Delta \eqdef 2^V \cup \bigcup_{w\in W} (\Delta_w\ast w)
\]
on the vertex set $V\sqcup W$ is $d$-collapsible.
\end{proposition}
\begin{proof}
For each $w$, fix a $(d-1)$-collapsing sequence of $\Delta_w$.
By adding $w$ to the free faces in this sequence, we obtain a sequence of $d$-collapses from $\Delta_w\ast w$ to $\Delta_w$.
Since the only faces of $\Delta$ that contain $w$ are those in $\Delta_w\ast w$, we may perform this sequence of $d$-collapses in $\Delta$ to remove all the faces containing $w$. 
After doing this for every $w\in W$, we are left with the simplex $2^V$, which is $d$-collapsible. 
Thus $\Delta$ is $d$-collapsible. 
\end{proof}

We use the preceding result to construct many $d$-collapsible complexes by starting with many $1$\nobreakdash-collapsible complexes.
\begin{collapsethmrestated}
The number $g_d(n)$ of $d$-collapsible complexes on the vertex set $[n]$ satisfies
  \[
  \tfrac{1}{(d+1)^{d+1}}n^{d+1} + O(n^d) \le \log_2 g_d(n) \le \tfrac{1}{(d+1)!} n^{d+1} + O(n^d). 
  \]
In particular, $g_d(n) = e^{\Theta(n^{d+1})}$. 
\end{collapsethmrestated}
\begin{proof}[Proof of the lower bound]
Since $g_d(n)$ is increasing in $n$, it suffices to prove that 
\begin{equation}\label{eq:collapselower_induct}
  \log_2 g_d(n)\geq \frac{1}{(d+1)^{d+1}}n^{d+1}\qquad\text{when }n\text{ is divisible by }d+1. 
\end{equation}
We prove this by induction on~$d$. When $d=1$, $g_d(n)$ is the number of labeled chordal graphs on the vertex set~$[n]$, 
since $1$\nobreakdash-collapsible complexes are exactly the clique complexes of chordal graphs.
Recall that a \emph{split graph} is a graph that can be partitioned into a clique and an independent set (with the edges between being arbitrary).
Split graphs are chordal, and one may construct $2^{n^2/4}$ split graphs on the vertex set $[n]$ by forming a clique on the first half of the vertices and adding arbitrary edges between the first and second 
halves. This proves \eqref{eq:collapselower_induct} when~$d=1$. 

Assume $d\ge 2$. Let $V$ be the first $\frac{d}{d+1}n$ vertices, and let $W$ be the remaining $\frac{1}{d+1}n$ vertices. 
For each vertex $w\in W$, let $\Delta_w$ be a $(d-1)$-collapsible complex on vertex set $V$.
Observe that we may recover $\Delta_w$ from the $d$-collapsible simplicial complex $\Delta$ defined in \Cref{thm:d-collapsible}, namely $\Delta_w$ is the link of $w$ in this complex.
Thus we obtain a unique $d$-collapsible complex on the vertex set $[n]$ for every choice of the various $\Delta_w$.
This implies that $g_d(n) \ge g_{d-1}(\tfrac{d}{d+1} n)^{n/(d+1)}$.
Taking logarithms and applying the inductive hypothesis, we obtain the lower bound as follows: 
\[
\log_2 g_d(n) \ge \log_2\left(g_{d-1}(\tfrac{d}{d+1} n)^{n/(d+1)}\right)
 \ge \tfrac{n}{d+1} \cdot \tfrac{1}{d^d} \tfrac{d^d}{(d+1)^d} n^d 
 = \tfrac{1}{(d+1)^{d+1}}n^{d+1}.\qedhere
\]
\end{proof}
\begin{proof}[Proof of the upper bound]
It is well-known (see for example the discussion in \cite{matousek_tancer_gap}) that $d$-collapsible complexes satisfy Helly's theorem,
and so are determined by their $d$-skeleta. Since the $d$-skeleton of an $n$-vertex complex can be specified by which of the  
$N\eqdef \binom{n}{d+1} +\binom{n}{d} + \cdots + \binom{n}{0}$ possible faces belong to it, it follows the number of $d$-collapsible
complexes on~$[n]$ is at most~$2^N$.
\end{proof}

\section{Problems and remarks}
\begin{itemize}
\item Bender, Richmond and Wormald \cite{bender_richmond_wormald} proved that almost all chordal graphs are split,
and so the lower bound on the number of $1$-collapsible complexes in \Cref{thm:d-collapsible} is tight. We suspect
that the lower bound remains tight also in higher dimensions, i.e., that $\log_2 g_d(n) = \frac{1}{(d+1)^{d+1}} n^{d+1} + O(n^d)$ for all $d\ge 1$. 
\item The lower bound in the Sandwich Theorem (\Cref{thm:order-sandwich} in \Cref{sec:interval}) can be improved  to $\tfrac{n(n-1)}{2} g(n-2)\le f_1(n)$. 
To obtain this bound one must first note that the interval graph we construct has an essentially unique representation by intervals, up to reflection about a common point. 
This follows from the fact that the interval graph has no ``buried subgraphs'' in the sense of Hanlon~\cite{hanlon}.
Thus even when we regard the graph as unlabeled, this unique representation allows us to recognize the vertices $n$ and $n-1$ as the outermost intervals, but not to tell these two vertices apart.
Hence, by choosing which two elements of $[n]$ label these vertices,
we can upgrade the injective map in the proof of Sandwich Theorem to a map
\[
 \{\text{Interval orders on $[n-2]$}\}\times \{(i,j)\in [n]^2\mid i\neq j\} \,\, \longrightarrow\,\, \{\text{Interval graphs on $[n]$}\}
\]
that is at most $2$-to-$1$.
\item Our results easily extend to the \emph{unlabeled} case. Indeed, the proof of Sandwich Theorem yields an injective map
\[
  \{\text{Unlabeled interval orders on $[n-2]$}\}\,\, \longrightarrow\,\, \{\text{Doubly-rooted interval graphs on $[n]$}\}.
\]
So, writing $\overline{f_d}(n)$ and $\overline{g}(n)$ for the unlabeled versions of $f_d(n)$ and $g(n)$, we obtain the inequality
$\frac{1}{n(n-1)}\overline{g}(n-2)\leq \overline{f_1}(n)\leq \overline{g}(n)$. As above, the factor of $\frac{1}{n(n-1)}$ can
be upgraded to $\frac{1}{2}$ by appealing to Hanlon's work~\cite{hanlon}. Either way, from the enumeration
of unlabeled interval orders \cite[Theorem~1]{brightwell_keller} we deduce that $\overline{f_1}(n)=e^{n\log n - (1+\log(\pi^2/6))n +O(\log n)}$.

Since the number of labeled and the number of unlabeled $d$-representable complexes differ by a factor of at most $n!$,
and $f_d(n)=e^{\Theta(n^d \log n)}$, it follows that $\overline{f_d}(n)=e^{\Theta(n^d \log n)}$ for $d\geq 2$ as well.
Similarly, the number of unlabeled $d$-collapsible complexes is $e^{\Theta(n^{d+1})}$ for $d\geq 1$, as per \Cref{thm:d-collapsible}. 

\item It is natural to consider, instead of the entire nerve, a skeleton thereof. Let $f_{d,k}(n)$ be the number of $k$-dimensional simplicial
  complexes that are $k$-skeletons of some $d$-representable complex on the vertex set $[n]$. This paper treated the case $k=d$. The case $(d,k)=(2,1)$
  was treated in \cite{pry}, where it is shown that almost almost every intersection graph of convex sets in the plane can be partitioned into $4$ parts
  such that $3$ of them induce a clique and the $4$th one splits into two cliques with no edge running between them.
  Using this, a computer algebra calculation suggests that
  $
    f_{2,1}(n)=\frac{2^{(3n^2+11)/8}}{3(n\ln 2)^{3/2}}\bigl(1+o(1)\bigr).
  $
  The case $d\geq 2k+1$ is trivial
  because every $k$-complex is a $k$-skeleton of some $d$-representable complex \cite{wegner_thesis,perelman} (see \cite{tancer_survey} for a proof in English).
  Hence, $f_{d,k}(n)=2^{\binom{n}{k+1}}$ for $d\geq 2k+1$.
  We do not know how $f_{d,k}$ behaves for any other pair $(d,k)$.

\item The constructions of large families of $d$-representable complexes (\Cref{prop:split-construction,prop:beta2}) suggest the following problem.
Say that a pair of convex sets $C$ and $C'$ are \emph{tangent to each other} if there is a point $p\in \partial C\cap \partial C'$
such that both $C$ and $C'$ are smooth at $p$, have the same tangent hyperplane at~$p$, but lie on the opposite sides of the hyperplane.
Both \Cref{prop:split-construction,prop:beta2} implicitly construct families of (not necessarily distinct) convex sets
with many tangencies. For example, \Cref{prop:beta2} constructs a family of $n$ convex sets in the plane whose tangency graph
is the complete $4$-partite graph with $\approx\tfrac{3}{4}\binom{n}{2}$ edges. We do not know if the maximum
number of tangencies among $n$ convex sets in~$\R^d$ is $\frac{d+1}{d+2}\binom{n}{2}+o(n^2)$.

\end{itemize}

\bibliographystyle{plain}
\bibliography{enumrepresentable.bib}

\begin{thebibliography}{10}

\bibitem{acan_chakraborty_jo_satti_succinct}
H\"{u}seyin Acan, Sankardeep Chakraborty, Seungbum Jo, and Srinivasa~Rao Satti.
\newblock Succinct encodings for families of interval graphs.
\newblock {\em Algorithmica}, 83(3):776--794, 2021.
\newblock \arXiv{1902.09228}.

\bibitem{bender_richmond_wormald}
E.~A. Bender, L.~B. Richmond, and N.~C. Wormald.
\newblock Almost all chordal graphs split.
\newblock {\em J. Austral. Math. Soc. Ser. A}, 38(2):214--221, 1985.

\bibitem{bcdk}
Mireille Bousquet-M\'{e}lou, Anders Claesson, Mark Dukes, and Sergey Kitaev.
\newblock {$(2+2)$}-free posets, ascent sequences and pattern avoiding
  permutations.
\newblock {\em J. Combin. Theory Ser. A}, 117(7):884--909, 2010.
\newblock \arXiv{0806.0666}.

\bibitem{brightwell_keller}
Graham Brightwell and Mitchel~T. Keller.
\newblock Asymptotic enumeration of labelled interval orders.
\newblock \arXiv{1111.6766}, November 2011.

\bibitem{cohen_komlos_mueller_probability}
Joel~E. Cohen, J\'{a}nos Koml\'{o}s, and Thomas Mueller.
\newblock The probability of an interval graph, and why it matters.
\newblock In {\em Relations between combinatorics and other parts of
  mathematics ({P}roc. {S}ympos. {P}ure {M}ath., {O}hio {S}tate {U}niv.,
  {C}olumbus, {O}hio, 1978)}, Proc. Sympos. Pure Math., XXXIV, pages 97--115
  (loose errata). Amer. Math. Soc., Providence, R.I., 1979.

\bibitem{eckhoff_vectors}
J\"urgen Eckhoff.
\newblock An upper-bound theorem for families of convex sets.
\newblock {\em Geometriae Dedicata}, 19:217--227, 1985.

\bibitem{gavoille_paul_labeling}
Cyril Gavoille and Christophe Paul.
\newblock Optimal distance labeling for interval graphs and related graph
  families.
\newblock {\em SIAM J. Discrete Math.}, 22(3):1239--1258, 2008.

\bibitem{hanlon}
Phil Hanlon.
\newblock Counting interval graphs.
\newblock {\em Trans. Amer. Math. Soc.}, 272(2):383--426, 1982.

\bibitem{kalai_intersection_patterns}
Gil Kalai.
\newblock Intersection patterns of convex sets.
\newblock {\em Israel J. Math.}, 48(2-3):161--174, 1984.

\bibitem{matousek_tancer_gap}
Ji\v{r}\'{\i} Matou\v{s}ek and Martin Tancer.
\newblock Dimension gaps between representability and collapsibility.
\newblock {\em Discrete Comput. Geom.}, 42(4):631--639, 2009.
\newblock \arXiv{0803.3520}.

\bibitem{mikos}
Patryk Mikos.
\newblock Efficient enumeration of non-isomorphic interval graphs.
\newblock {\em Discrete Math. Theor. Comput. Sci.}, 23(1):Paper No. 2, 23,
  2021.
\newblock \arXiv{1906.04094}.

\bibitem{pry}
J\'{a}nos Pach, Bruce Reed, and Yelena Yuditsky.
\newblock Almost all string graphs are intersection graphs of plane convex
  sets.
\newblock {\em Discrete Comput. Geom.}, 63(4):888--917, 2020.

\bibitem{perelman}
G.~Ya. {Perel'man}.
\newblock Realization of abstract {$k$}-skeletons as {$k$}-skeletons of
  intersections of convex polyhedra in {${\bf R}^{2k-1}$}.
\newblock In {\em Geometric questions in the theory of functions and sets},
  pages 129--131. Kalinin. Gos. Univ., Kalinin, 1985.

\bibitem{tancer_projective_planes}
Martin Tancer.
\newblock Non-representability of finite projective planes by convex sets.
\newblock {\em Proceedings of the American Mathematical Society},
  138(9):3285--3291, 2010.

\bibitem{tancer_survey}
Martin Tancer.
\newblock Intersection patterns of convex sets via simplicial complexes: a
  survey.
\newblock In {\em Thirty essays on geometric graph theory}, pages 521--540.
  Springer, New York, 2013.
\newblock \arXiv{1102.0417}.

\bibitem{tancer_np_complete}
Martin Tancer.
\newblock Recognition of collapsible complexes is {NP}-complete.
\newblock {\em Discrete and Computational Geometry}, 55:21--38, 2016.
\newblock \arXiv{1211.6254}.

\bibitem{wegner_thesis}
Gerd Wegner.
\newblock {\em Eigenschaften der Nerven homologisch-einfacher Familien im
  {$\R^n$}}.
\newblock PhD thesis, Universit\"at G\"ottingen, 1967.

\bibitem{wegner_collapsing}
Gerd Wegner.
\newblock {$d$}-collapsing and nerves of families of convex sets.
\newblock {\em Archiv der Mathematik}, 26:317--321, 1975.

\bibitem{ysku}
Kazuaki Yamazaki, Toshiki Saitoh, Masashi Kiyomi, and Ryuhei Uehara.
\newblock Enumeration of nonisomorphic interval graphs and nonisomorphic
  permutation graphs.
\newblock {\em Theoret. Comput. Sci.}, 806:310--322, 2020.

\bibitem{yang_pippenger}
Joyce~C. Yang and Nicholas Pippenger.
\newblock On the enumeration of interval graphs.
\newblock {\em Proc. Amer. Math. Soc. Ser. B}, 4:1--3, 2017.
\newblock \arXiv{1609.02479}.

\bibitem{zagier}
Don Zagier.
\newblock Vassiliev invariants and a strange identity related to the {D}edekind
  eta-function.
\newblock {\em Topology}, 40(5):945--960, 2001.

\end{thebibliography}
\end{document}